\documentclass[11pt]{amsart}
\usepackage{pstricks}
\usepackage{color}
\usepackage{amssymb}

\usepackage{appendix}
\usepackage{longtable}
\usepackage{setspace}

\usepackage{pstricks}
\usepackage{color}

\usepackage[all,cmtip]{xy}

\usepackage{graphicx}
\usepackage{xypic}

\theoremstyle{plain}
\newtheorem{thm}{Theorem}[section]
\newtheorem{theorem}[thm]{Theorem}

\newtheorem{lemma}[thm]{Lemma}

\newtheorem{set-up}[thm]{Set-up}
\theoremstyle{definition}
\newtheorem{remark}[thm]{Remark}

\newtheorem{example}[thm]{Example}

\numberwithin{equation}{section}

\newcommand{\sA}{{\mathcal A}}
\newcommand{\sB}{{\mathcal B}}
\newcommand{\sC}{{\mathcal C}}
\newcommand{\sD}{{\mathcal D}}

\newcommand{\sF}{{\mathcal F}}

\newcommand{\sI}{{\mathcal I}}

\newcommand{\sM}{{\mathcal M}}

\newcommand{\sP}{{\mathcal P}}

\newcommand{\sR}{{\mathcal R}}
\newcommand{\sS}{{\mathcal S}}

\newcommand{\sZ}{{\mathcal Z}}

\newcommand{\C}{{\mathbb C}}

\newcommand{\Q}{{\mathbb Q}}
\newcommand{\R}{{\mathbb R}}

\newcommand{\Z}{{\mathbb Z}}

\title [K3 surface entropy and automorphism groups]{K3 surface entropy and automorphism groups}

\author{Xun Yu}
\address{Center for Applied Mathematics, Tianjin University, 92 Weijin Road, Nankai District,
Tianjin 300072, P. R. China.
}
\email{xunyu@tju.edu.cn}

\begin{document}
\begin{abstract} 
We derive a characterization of the complex projective K3 surfaces which have automorphisms of positive entropy in term of their N\'eron-Severi lattices. Along the way, we classify the projective K3 surfaces of zero entropy with infinite automorphism groups and we determine the projective K3 surfaces of Picard number at least five with almost abelian automorphism groups, which gives an answer to a long standing question of Nikulin.
\end{abstract}
\maketitle

\setcounter{tocdepth}{1}
\tableofcontents


\section{Introduction}
A {\it K3 surface} is a compact complex manifold $X$ of dimension $2$ with trivial canonical bundle and $H^1(X,\mathcal{O}_X)=0$. For an automorphism $f\in {\rm Aut}(X)$,  it is known that the topological entropy $h(f)$ of $f$ is determined by $$h(f)={\rm log}\; \lambda_1 (f).$$ Here $\lambda_1(f)$ is the {\it first dynamical degree of} $f$, i.e., the spectral radius of  $f^*| H^2(X,\mathbb{C})$, which coincides with the spectral radius of the action $f^*| {\rm NS}(X)$ on the N\'eron-Severi lattice ${\rm NS}(X)$  when $X$ is projective (see e.g. \cite[Corollary 1.4]{ES13}). If $h(f)>0$ (resp. $h(f)=0$), then we say  $f$ is {\it of positive entropy} (resp. {\it of zero entropy}). If there exists (resp. does not exist) $f\in {\rm Aut}(X)$ with $h(f)>0$, then we say $X$ is {\it of positive entropy} (resp. {\it of zero entropy}). A {\it Salem number} $\tau >1$ is an algebraic integer which is conjugate to $\frac{1}{\tau}$, and whose remaining conjugates lie on the unit cycle $S^1$. A {\it Salem polynomial} is the minimal polynomial of a Salem number. See \cite{BHPV04} (resp. \cite{Mc02}, \cite{DS05}) for basics on complex surfaces (resp. complex dynamics) that we use here. 

 Cantat \cite[Proposition 1]{Ca99} observed that a compact K\"ahler surface of positive entropy is bimeromorphic to one of the following surfaces: a K3 surface, a complex torus of dimension $2$, an Enriques surface, or $\mathbb{P}^2$. McMullen \cite{Mc02} observed that if $h(f)>0$, then the characteristic polynomial of $f^*| H^2(X, \mathbb{C})$ is the product of one Salem polynomial and some (possibly no) cyclotomic polynomials (thus, the first dynamical degree $\lambda_1(f)$ is a Salem number if $h(f)>0$), and he constructed (non-projective) K3 surface automorphisms with Siegel disks. The two observations highlight unexpected relations among surface automorphisms, complex dynamics, and number theory. Such relations have caught the attention of many people from various points of view (see \cite{Og14} for a survey on related studies). Interesting examples of projective K3 surface automorphisms of positive entropy have been found (see e.g. \cite{Mc11a}, \cite{BK14}, \cite{Og15}, \cite{Mc16}, \cite{BE21}). Lehmer number ($=1.176280818...$) is the smallest Salem number which can be realized by (rational) surface automorphisms (\cite{Mc07}) but it cannot be realized on $2$-dimensional complex tori (for trivial reasons) or Enriques surfaces (\cite[Theorem 1.2]{Og10a}). McMullen \cite[Theorem 1.1]{Mc16} proved that it can be realized by an automorphism of a projective K3 surface. Based on \cite{Mc16}, Brandhorst-Elkies \cite[Theorem 1.2]{BE21} derived explicit equations for the surface and automorphism. Moreover, interesting surfaces of other types might be obtained from projective K3 surfaces of positive entropy (see e.g. \cite{OY20}, \cite{DOY21}). However, it seems still unknown which projective K3 surfaces are of positive entropy. Our aim here is to characterize such surfaces in terms of their N\'eron-Severi lattices. 

From now on, let $X$ be a projective K3 surface. See Section \ref{sec:lattices} for basics on lattices. The surface $X$ being of positive entropy or not depends only on the isometry class of the even hyperbolic lattice ${\rm NS}(X)$ since the natural homomorphism
$$
{\rm Aut}(X)\longrightarrow {\rm O}({\rm NS}(X))/W({\rm NS}(X))
$$
has finite kernel and cokernel by a consequence of the Torelli theorem for algebraic K3 surfaces (\cite{PS71}). If a {\it genus one fibration} $\varphi: X\longrightarrow \mathbb{P}^1$ (i.e., a general fiber of the morphism $\varphi$ is an elliptic curve) admits a section, then we call $\varphi$ an {\it elliptic fibration}. We denote by $\rho(X):={\rm rk}( {\rm NS}(X))$ the Picard number of $X$. For $X$ of zero entropy, ${\rm Aut}(X)$ could be either finite or infinite. Let
$$
\sF:= \{ {\rm NS}(X) |\; X \text{ is a projective K3 surface with }|{\rm Aut}(X)|<\infty \text{ and } \rho(X)\ge 3\},
$$
and
$$
\sI:= \{ {\rm NS}(X) |\; X \text{ is a projective K3 surface of zero entropy with }|{\rm Aut}(X)|=\infty\}.
$$
Our aim is equivalent to determination of the sets $\sF$ and $\sI$ (modulo cases $\rho(X)\le 2$).  The set $\sF$  consists of exactly $118$ lattices by classification for $\rho(X)\neq 4$ due to Nikulin \cite{Ni81}, \cite{Ni85} and by classification for $\rho(X)=4$ due to Vinberg \cite{Vi07}. For $|{\rm Aut}(X)|=\infty$, Oguiso \cite[Theorem 1.4]{Og07} observed that $X$ is of zero entropy if and only if ${\rm Aut}(X)$ preserves a unique genus $1$ fibration; if $X$ is of zero entropy, then ${\rm Aut}(X)$ is almost abelian. Thus,  
$$
\sI=\{{\rm NS}(X)|\; X\text{ proj. K3 with infinite }{\rm Aut}(X)\text{ preserving a unique genus $1$ fibration}\}. 
$$ 
Hence, people may study $\sI$ from two different motivations: entropy and automorphism groups. The lattices in $\sI$ are closely related to $\sF$ and they behave like $U\oplus \Lambda$ (see \cite{Bo90}), where $\Lambda$ is the Leech lattice. Thus, $\sI$ is interesting from arithmetic, dynamical and geometric points of view. Actually, attempts to find members of $\sI$ have a long history (see \cite{Ni79}, \cite{Ni81}, \cite{Ni96}, \cite{Ni99}, \cite{Og07}, \cite{Ni14}, \cite{OY19}, \cite{Ni20}, \cite{Me21}, etc.). 

Next, we briefly review known results on $\sI$. Nikulin (\cite{Ni96}, \cite{Ni14}) observed that $\sI$ is a finite set and asked for classification of the lattices in $\sI$ (i.e., $\sS {\rm EK3}_p$ in \cite[Section 4]{Ni14} when the base field is $\C$). Note that $1\le \rho(X)\le 20$. If $\rho(X)\le 2$, then ${\rm NS}(X)\notin \sI$. So one only needs to consider the cases $3\le \rho(X)\le 20$.  Let 
$$
{\sI}^ {n}:=\{{\rm NS}(X)\in \sI | \; {\rm rk}({\rm NS}(X))=n\}.
$$
Nikulin (\cite{Ni81}, \cite{Ni20}) obtained twelve $2$-elementary lattices in $\sI$  by considering fixed points of certain involutions and there is no other $2$-elementary lattice in $\sI$ (see also \cite{OY19}; for a few $p$-elementary ($p=3,5,7$) ${\rm NS}(X)\in \sI$ which can be proved similarly, see \cite{AST11}). Nikulin \cite[Section 3.3]{Ni99} found the $18$ rank $3$ lattices in $\sI$ using Vinberg's algorithm \cite{Vi75}. Oguiso \cite[Theorem 1.6]{Og07} proved that 
${\sI}^{20}$ is empty based on \cite{SI77}. With the help of computer, Mezzedimi \cite[Theorems 0.2, 0.3]{Me21} proved that ${\sI}^{19}$ is empty and the subset 
$\{{\rm NS}(X)\in \sI |\; {\rm Aut}(X) \text{ preserves}$ $\text{an elliptic fibration with only irreducible fibers} \}
$ consists of $32$ lattices. To the best of our knowledge, all known examples (about $50$ lattices) in $\sI$ have been reviewed above. To sum up, the lattices in $\sI$ has been classified only for a few Picard numbers and for lattices satisfying some technical conditions. Thus, the classification problem for $\sI$ is still widely open. Now we solve this problem completely by finding all the remaining lattices, which gives an answer to the question of Nikulin on $\sS {\rm EK3}_p$ in \cite[Section 4]{Ni14}.

\begin{theorem}\label{main1} The set $\sI$ consists of exactly the $193$ lattices listed in the Appendix.\end{theorem}

In particular, many members of $\sI$ are new. We denote by
$$
\sZ:= \sF \cup \sI.
$$ 
By Theorem \ref{main1} and the known results on $\sF$, we obtain our desired characterization of projective K3 surfaces of positive entropy.
\begin{theorem}\label{main2}
 Let $X$ be a projective K3 surface.  Then $X$ is of positive entropy if and only if one of the following statements holds: 
 
 \begin{enumerate}
\item $\rho(X)\ge 3$ and ${\rm NS}(X)$ is isometric to none of the $311$ lattices in $\sZ$;
\item $\rho(X)=2$ and  there exists no nonzero $x\in {\rm NS}(X)$ with $x^2=0, -2$.
\end{enumerate}
\end{theorem}
For cases $\rho(X)\le 2$, see e.g. \cite[Chapter 15, Example 2.11]{Hu16}. Note that the set

$$
{\sZ}^ {\geqslant 5}:=\{{\rm NS}(X)\in \sZ | \; {\rm rk}({\rm NS}(X))\ge 5\}
$$ 
has exactly $229$ lattices. Moreover, by \cite[Theorem 9.1.1]{Ni81} and \cite[Theorem 1.4]{Og07}, the following $3$ conditions are equivalent when $|{\rm Aut}(X)|=\infty$ and $\rho(X)\ge 5$: $X$ of zero entropy, ${\rm Aut}(X)$ preserving a unique genus $1$ fibration, ${\rm Aut}(X)$ almost abelian. Combining with Theorem \ref{main2}, we have the following result which gives an answer to Nikulin's question on classification of projective K3 surfaces with almost abelian automorphism groups (see \cite[Remark 9.1.2]{Ni81}).

\begin{theorem}\label{main3}
Let $X$ be a projective K3 surface with $\rho(X)\ge 5$. Then ${\rm Aut}(X)$ is almost abelian if and only if ${\rm NS}(X)$ is isometric to one of the $229$ lattices in ${\sZ}^ {\geqslant 5}$.
\end{theorem}

Next we briefly explain ideas of the (computer-assisted) proof of Theorem \ref{main1}. Roughly speaking, we first find finitely many explicit candidates $L$ and then check whether $L\in \sI$ one by one. Slightly more precisely, there are two cases: (1) $L$ are of the form $U\oplus M$ and (2) $L\neq U\oplus M$. Here $U$ is the lattice with Gram matrix $\small{\left(\begin{array}{cc} 
0 & 1 \\
1 & 0 
\end{array} \right)}$ and $M$ are even negative definite lattices. Two issues arise in our approach: 
(i) for case (1), an explicit finite list of candidates of $L$ is known but there are thousands of candidates; (ii) it is unclear how to handle case (2) (for example, unlike (1), an explicit list of finitely many candidates is unknown). We introduce a new notion, the {\it critical sublattice} (Section \ref{sec:criticalsub}). Critical sublattices are particularly important for us since they play crucial role in dealing with both issues (i) and (ii). Due to our characterization of critical sublattices (Theorem \ref{thm:S}), for case (2), all we need to do is (essentially) to compute the critical sublattices $L_{cr}$ of $L=U\oplus M\in \sI$ (see also Lemma \ref{lem:overlattice}). Moreover, by properties of $L_{cr}$ (Lemma \ref{lem:S}, Theorems \ref{thm:S=L}, \ref{thm:SinL}), such computation can be done directly even without using computer algebra (see Example \ref{exm:rk3}). In fact, it turns out that Theorem \ref{thm:S=L} can apply to many lattices $L=U\oplus M\in \sI$ and then we get $L_{cr}=L$ immediately without doing any further calculations. Hence, case (2) (in particular, issue (ii)) is completely solved by critical sublattices. Using critical sublattices again, we prove an explicit relation between $\sI$ and $\sF$ (Theorem \ref{thm:overlattice}), based on which we introduce a new test, the {\it overlattice test} (Section \ref{sec:tests}), for checking entropy-positivity of lattices $L=U\oplus M$. For case (1), the overlattice test rules out many candidates in an early stage (see e.g. Theorem \ref{thm:atleast13}), which reduces the amount of the work needed for further calculations considerably. Combining this test with (generalizations of) other (known) tests (see Remarks \ref{rem:tests}, \ref{rem:exception}), we find all the lattices $U\oplus M\in\sI$ using computer algebra, which means we solve case (1).   

Let us explain a little more about Theorems \ref{thm:S}, \ref{thm:S=L}, \ref{thm:SinL}, \ref{thm:overlattice} on critical sublattices. The proofs of them are entirely free from computer algebra. However, our resulting formulations are ones which fit well with practical calculations. It is expected that lattices in $\sI$ and $\sF$ are closely related but an explicit relation like Theorem \ref{thm:overlattice} seems rare. We believe that these results are interesting in and of themselves and will be applicable to other problems.

We conclude the Introduction by posing some open questions closely related to our work. A systematic way to check realizability of a prescribed Salem number by   projective K3 surface automorphisms has been established by McMullen (\cite{Mc02}, \cite{Mc11a}, \cite{Mc16}). A closely related problem in a somewhat inverse direction is the following: for a prescribed projective K3 surface $X$ of positive entropy, it would be interesting if there is a way to determine (minimum of) the Salem numbers which can be realized on $X$.  This problem might be tractable for such $X$ with an explicit (finite) set of generators of ${\rm Aut}(X)$ (see e.g. \cite{Vi83}, \cite{KK01}, \cite{DK02}, \cite{Sh15}, \cite{HKL20}).  The groups ${\rm Aut}(X)$ for ${\rm NS}(X)\in\sF$ are small (\cite{Ko89}, see e.g. \cite{Mu88} for larger finite groups acting on K3 surfaces) and the geometry of such $X$ is well-understood (see e.g. \cite{Ni81}, \cite{Ro20}).  It is of interest to compute ${\rm Aut}(X)$ for ${\rm NS}(X)\in \sI$ and to study the geometry of such $X$. In Sections \ref{sec:lattices}-\ref{sec:tests}, everything is stated in lattice language. In particular, the arguments/results there are characteristic free and are applicable to classification for analogue of $\sI$ in positive characteristic.

\medskip

{\bf Acknowledgement.} I would like to thank Professor Keiji Oguiso for valuable discussions and comments. This work is partially supported by the National Natural Science Foundation of China (No. 12071337, No. 11701413, No. 11831013).


\section{Lattices and zero entropy triples}\label{sec:lattices}
In this section, we recall some basics on lattices (for more details, we refer to \cite{Ni80}, \cite{CS99}, \cite{Mc16}) and we introduce the notion of zero entropy triple which will be used later.

\subsection*{Lattices} A {\it lattice} is a finitely generated free $\Z$-module  $L$ together with a symmetric bilinear form $(*, **)=(*, **)_L: L\times L\longrightarrow \Z.$ We often denote $(x,x)$ by $x^2$ for simplicity. We say $L$ is an {\it even} lattice if $x^2\in 2\Z$ for all $x\in L$; otherwise, $L$ is called an {\it odd} lattice. We call $x\in L$ a {\it root} if $x^2=-2$. A {\it root lattice} is a lattice generated by roots. By $A_i$ ($i\ge 1$), $D_j$ ($j\ge 4$), $E_k$ ($k=6,7,8$) we mean the negative definite root lattice whose basis is given by the corresponding Dynkin diagram. We call $L$ a {\it hyperbolic} lattice of rank ${\rm rk}(L)=n$ if the signature of $L$ is $(1, n-1)$, where the first (resp. second) entry is the number of positive (resp. negative)
 squares. We denote by $U$ the unique rank 2 unimodular hyperbolic even lattice. For a sublattice $K\subset L$ (resp., an element $x\in L$), we denote by $K_L^\perp$ (resp., $x_L^\perp$) the orthogonal complement (the subscript $L$ is sometimes omitted if there is no confusion).  The {\it determinant} ${\rm det}(L)$ of a lattice $L$ is the determinant of any Gram matrix of $L$. For a nonzero $a\in \Z$, the lattice $L(a)$ is defined to be the same $\Z$-module as $L$ with the bilinear form $(x,y)_{L(a)}:=a(x,y)_L.$ We use $[a]$ to denote the rank one lattice with determinant $a$. 
 
\subsection*{Glue groups and discriminant forms} Let $L$, $L^\prime$ be two nondegenerate even lattices. We may view its dual $L^*={\rm Hom}(L,\Z)$ as a subset of $L\otimes \Q$ via the canonical embedding $L\hookrightarrow L^*$ determined by the bilinear form of $L$. The finite abelian quotient group $G(L):=L^*/L$ is called the {\it glue group} of $L$. For any $x\in L^*$, we denote by $\overline{x}$ the image of $x$ in $G(L)$ under the natural quotient map. We use $l(L)$ to denote the minimum number of generators of $G(L)$. The bilinear form of $L$ induces a ($\Q$-valued) bilinear form on $L^*$, which gives a bilinear form $b_L$ and a quadratic form $q_L$ on the glue group $G(L)$ as 
$$
b_L: G(L)\times G(L) \longrightarrow \Q/\Z, \; b_L(\overline{x},\overline{y})\equiv (x, y) \; {\rm mod}\; \Z,$$
and $$
q_L: G(L)\longrightarrow \Q/2\Z, \;\; q_L(\overline{x})\equiv (x,x)\; {\rm mod}\; 2\Z. $$
We say that $q_L$ is the {\it discriminant form}
 of $L$. Any isometry $f\in {\rm O}(L)$ of $L$ naturally induces an automorphism $\overline{f}\in {\rm O}(q_L)$ of the discriminant form $q_L$. If $L\subset L^\prime$ and they have the same rank,  then we say that $L^\prime$ is an {\it overlattice} of $L$. In such case, the subgroup $H_{L^\prime}:=L^\prime/L\subset G(L)$ is isotropic in the sense that $b_L| H_{L^\prime}=0$.  Overlattices (resp. even overlattices) of $L$ correspond bijectively to isotropic subgroups $H\subset G(L)$ (resp. $H\subset G(L)$ with $q_L | H=0$). A {\it gluing map} $\phi: H_1\longrightarrow H_2$ is an isomorphism of abelian groups $H_1\subset G(L)$ and $H_2\subset G(L^\prime)$ with $\phi (x, y)=-(x,y),$ for all $x,y \in H_1$. For a prime number $p$, we say that $L$ is a {\it $p$-elementary} lattice if $G(L)\cong {(\Z/p\Z)}^{l(L)}$.
 
 Let $L_i$ ($i=1,2$) be two nondegenerate even lattices of the same rank and signature. We say that $L_1$ and $L_2$ are in the same {\it genus} if their discriminant forms $q_{L_i}$ are isomorphic to each other. By the {\it genus of} $L$ we mean the set of isometry classes of lattices in the genus of $L$. The genus of $L$ is called a {\it one-class} genus if it contains only $L$. Genera of indefinite lattices are often one-class genera (cf. \cite[Section $1.13^\circ$]{Ni80}, \cite[Chapter 15]{CS99}). 
   
\subsection*{Chambers and Weyl groups} Let $L$ be an even hyperbolic lattice and $$\sR_{L}:=\{x\in L|\;  x^2=-2\}.$$ The positive cone $\mathcal{P}_{L}$ of $L$ is defined to be one of the two connected components of $$\{x\in L\otimes \R |\; x^2 >0\}.$$ We denote by ${\rm O}^+(L)\subset {\rm O}(L)$ the subgroup consisting of isometries preserving $\mathcal{P}_{L}$. The {\it Weyl group} $$W(L)\subset {\rm O}^+(L)$$ of $L$ is the subgroup generated by all the reflections $$s_v: L\longrightarrow L,\,\, s_v(x)=x+(x,v)v, \, \forall\, x\in L,$$ where $v\in \sR_L$. The positive cone $\mathcal{P}_L$ is cut into open, convex {\it chambers} $\sC$ by the set of all {\it root hyperplanes} $$v^\perp:=v_{L\otimes \R}^{\perp}=\{x\in L\otimes \R|\; (x,v)=0\},$$ where $v\in \sR_L$. The Weyl group $W(L)$ acts simply transitively on the chambers of $\sP_L$. 

Let $\sC$ be a chamber of $\sP_L$. If a nonempty open subset of a root hyperplane $r^\perp$ is contained in the closure $\overline{\sC}$ (inside $L\otimes \R$) of $\sC$, then $r^\perp$ is called a {\it wall} of $\sC$. Let $h\in L\cap \sC$. Note that $h_L^\perp$ contains no roots. We call a root $r\in L$ a {\it $\sC$-irreducible root} if the root hyperplane $r^\perp$ is a wall of $\sC$ and $(r,h)>0$. Any root in $L$ can be written as a linear combination of $\sC$-irreducible roots with integer coefficients of the same sign (cf. \cite[Theorem 1.1]{Bo90}). The {\it automorphism group} $${\rm Aut}(\sC)\subset {\rm O}^+(L)$$ of $\sC$ is the group of isometries of $L$ preserving $\sC$. The quotient group ${\rm O}^+(L)/W(L)$ is isomorphic to ${\rm Aut}(\sC)$. The lattice $L$ is called {\it $2$-reflective} if ${\rm O}^+(L)/W(L)$ is a finite group.

\subsection*{Zero entropy lattices and triples} Let $L$ be an even hyperbolic lattice and let $\sC$ be a chamber of the positive cone $\sP_L$.  For $f\in {\rm Aut}(\sC)\subset {\rm O}^+(L)$, the characteristic polynomial $\chi_f(x)$ of $f$ is the product of cyclotomic polynomials and at most one Salem polynomial counted with multiplicities (see \cite{Mc02}, \cite{Og10b}). If a Salem polynomial is a factor of $\chi_f(x)$, then we call $f$ {\it of positive entropy}; otherwise, $f$ is called {\it of zero entropy}. The lattice $L$ is {\it of zero entropy} if $f$ is of zero entropy for all $f\in {\rm Aut}(\sC)$; otherwise, $L$ is called {\it of positive entropy}. If $L$ is an overlattice of a zero entropy lattice $L_1$, then $L$ is of zero entropy too. In fact, if $f$ is an isometry of $L$ preserving a chamber of $\sP_L$, then a sufficiently large power of $f$ preserves $L_1$ and a chamber of $\sP_{L_1}$. For $x\in L$, we set $${\rm Aut}(\sC)_x:=\{f\in {\rm Aut}(\sC)|\; f(x)=x\}.$$

\begin{theorem}[{\cite{Og07}}]\label{thm:Og}
Let $L$ be an even hyperbolic lattice of zero entropy  with infinite automorphism group ${\rm Aut}(\sC)$, where $\sC\subset \sP_L$ is a chamber. Then there exists a unique primitive isotropic element $e\in L\cap (\overline{\sC}\setminus \{0\})$ such that ${\rm Aut}(\sC)$ preserves $e$. Moreover, $e$ is the unique primitive isotropic element $e\in L\cap (\overline{\sC}\setminus \{0\})$ such that ${\rm Aut}(\sC)_e$ is infinite.
\end{theorem}

Let $L$, $\sC$, $e$ be as in Theorem \ref{thm:Og}. Then we say $(L, \sC, e)$ is a {\it zero entropy triple}. Note that the quotient $e_L^\perp/\langle e\rangle$  is a well-defined negative definite even lattice which has no finite index root sublattices (\cite[Corollary 1.5.4]{Ni81}).  If $d\in L$ is a $\sC$-irreducible root with $(d,e)=1$, then we say $d$ is a {\it $\sC$-section} of $e$.  An even hyperbolic lattice containing $U$ is not $2$-reflective if and only if the genus of $U^\perp$ contains a lattice which has no root sublattices of finite index (\cite[Theorem 10.2.1]{Ni81}). Thus, $e$ has a $\sC$-section if and only if $U\subset L$. There are only finitely many zero entropy triples up to isometry (\cite{Ni96}). 

The following lemma will be useful later. In fact, it reduces our classification in Section \ref{sec:proof} to consideration of lattices of two special types:  lattices $L$ containing $U$ (to such lattices the tests in Section \ref{sec:tests} can be applied effectively) and their critical sublattices $L_{cr}$ (which can be easily handled based on results in Section \ref{sec:criticalsub}). 

\begin{lemma}\label{lem:overlattice}
Let $(L, \sC, e)$ be a zero entropy triple. If $L$ does not contain $U$, then there exists a zero entropy triple $(L^\prime, \sC^\prime, e^\prime)$ such that $U\subset L^\prime$ and $L^\prime$ is an overlattice of $L$.
\end{lemma}

\begin{proof}
Since $L$ does not contain $U$ and $e\in L$ is primitive  and isotropic, it follows that $\{(e, x)| \; x \in L\}=l \,\Z$ for some integer $l>1$. Then $L^\prime:=\langle \frac{e}{l}, L\rangle\subset L^*$ is a well-defined even overlattice of $L$ and $e^\prime := \frac{e}{l}$ is a primitive isotropic element with $(e^\prime, y)=1$ for some $y\in L^\prime$ (cf. \cite[Lemma 2.1]{Ke00}). Thus, $L^\prime $ contains $\langle e^\prime, y\rangle=U$ and $U_{L^\prime}^\perp\cong e_L^\perp/\langle e\rangle$ has no finite index root sublattices. Hence, $L^\prime$ is not $2$-reflective and $(L^\prime, \sC^\prime, e^\prime)$ is a zero entropy triple for some chamber $\sC^\prime$.
\end{proof}


\section{Critical sublattices}\label{sec:criticalsub}
Let $(L, \sC, e)$ be a zero entropy triple and $n:={\rm rk}(L)$. We define the {\it critical} sublattice of $L$, denoted by $L_{cr}$, to be the intersection of all rank $n$ zero entropy sublattices of $L$. We set 
$$\sS_L:=\{w(c) |\; c \text{ is a }\sC \text{-irreducible root with }(c,e)>0, w\in W(L)\}.$$
In this  section, we show that the critical sublattice $L_{cr}\subset L$ is generated by $\sS_L$ and $L_{cr}$ is of zero entropy (Theorem \ref{thm:S}). Moreover, we derive some consequences (Theorems \ref{thm:S=L}, \ref{thm:SinL}, \ref{thm:overlattice}) which are frequently used during our classification for $\sI$.

\begin{lemma}\label{lem:eC}
Let $L_1\subset L$ be a sublattice of rank $n$. If $\sS_L\cap L_1\neq \sS_L$, then $L_1$ is of positive entropy. \end{lemma}

\begin{proof}
By the assumption, there exists $w\in W(L)$ and a $\sC$-irreducible root $c$ such that $(c,e)>0$ and $w(c)\notin L_1$. Note that it suffices to consider the case $w={\rm id}_{L}$. Let $h\in L\cap \sC$. Recall that $\sC$ is a chamber of the positive cone $\sP_L$. Since the root hyperplane $c^\perp$ is a wall of $\sC$, there exists a chamber $\sC_1\subset \sP_L$ and $h_1\in L\cap\sC_1$ such that 
$$\{r\in L|\, r^2=-2, (h,r)(h_1,r)<0\}=\{\pm c\}.$$
Then the reflection $s_{c}\in W(L)$ defined by $c$ satisfies $s_c(\sC)=\sC_1$. Since $e\in\overline{\sC}$, we have $$s_c(e)\in s_c(\overline{\sC})=\overline{\sC_1}.$$ Since $L$ is an overlattice of $L_1$, there exist positive integers $N_1, N_2$ such that $N_1 e$ and $N_2 s_c(e)$ are primitive elements in $L_1$. Since $c\notin L_1$, it follows that $\sC$ and $\sC_1$ are contained in the same chamber, say $\sD$, of the positive cone $\sP_{L_1}$. Then, $N_1e$ and $N_2 s_c(e)$ are two different primitive isotropic elements contained in $L_1\cap(\overline{\sD}\setminus \{0\})$ such that ${\rm Aut}(\sD)_{N_1 e}$ and  ${\rm Aut}(\sD)_{N_2 s_c(e)}$ are infinite (see \cite[Lemmas 2.1, 5.4]{Yu18}). By Theorem \ref{thm:Og}, $L_1$ is of positive entropy.\end{proof}

The following characterization of critical sublattices plays a key role in our classification for $\sI$.
\begin{theorem}\label{thm:S}
Let $S\subset L$ be the sublattice generated $\sS_L$ and let $L_1\subset L$ be a sublattice of rank $n$. Then $L_1$ is of zero entropy if and only if $S\subset L_1$. In particular, $L_{cr}=S$.
\end{theorem}

\begin{proof}
Clearly the Weyl group $W(L)$ preserves $\sS_L$ and $S$. If $L_1$ is of zero entropy, then by Lemma \ref{lem:eC}, $L_1$ contains $\sS_L$ and $S$. Suppose $S\subset L_1$. Then it suffices to show that $S$ is of zero entropy since $L_1$ is an overlattice of $S$. Suppose $S$ is of positive entropy.  Then there exists $f_S\in {\rm O}^+(S)$ preserving a chamber $\sD\subset \sP_S$ such that $f_S$ is of positive entropy. Since $L$ is an overlattice of $S$, replacing $f_S$ by its sufficiently large power, we may assume that $f_S$ extends to an isometry $f_L$ of $L$, i.e., $$f_L\in {\rm O}^+(L) \;\text{ and }\; f_L|S=f_S.$$ There exists a positive integer $N$ such that $Ne$ is a primitive element in $S$. Note that $f_S(Ne)\neq Ne.$ Let $$h\in S\cap L\cap \sC.$$ Without loss of generality, we may assume $h$ is contained in the chamber $\sD$. Then $f_S(h)$ (resp. $f_S(Ne)$) is contained in $\sD$ (resp. $\overline{\sD}$).  Since $L$ is of zero entropy, it follows that $f_L\in {\rm O}^+(L)$ does not preserve $\sC$. Thus, the set
$$\sA:=\{c\in L|\, c\text{ is a }\sC\text{-irreducible root with }(c,f_L(h))<0\}$$
is not empty. Note that if $c\in L$ is a $\sC$-irreducible root such that $(c,e)>0$, then $$c\in\sS_L\subset S \;\text{ and }\; (c, h)(c,f_S(h))>0,$$ which implies $(c, f_L(h))>0$. Thus, $$\sA\subset \sB:=\{c\in L|\, c\text{ is a }\sC\text{-irreducible root with }(c,e)=0\}.$$
Let $c_1\in \sA$. Then $a:=(c_1,h)>0$ and $b:=(c_1,f_L(h))<0$. Let $$h_1:=a f_L(h)-b h.$$ Then $h_1\in \sD\cap S$. Consider the reflection $s_{c_1}\in W(L)$. Thus, we have $s_{c_1} (S)=S$. Note that $s_{c_1}(h_1)=h_1.$ Then $s_{c_1}|S$ preserves the chamber $\sD$ and we have $$s_{c_1}(f_L(h))\in\sD\cap S.$$ If $s_{c_1}(f_L(h))\notin \sC$, then by similar arguments to find $c_1$, we find $c_2\in\sB$ such that $$(c_2, s_{c_1}(f_L(h)))<0 \;\text{ and } \; (s_{c_2}\circ s_{c_1})(f_L(h))\in \sD\cap S.$$ By repeating the above process, eventually we find a sequence of elements $c_1,c_2,...,c_m\in\sB$ such that $$(s_{c_m}\circ \cdots\circ s_{c_{2}}\circ s_{c_1})(f_L(h))\in \sC$$ (cf. \cite[Section 6, Proof of Theorem 1 3)]{PS71}). Then $$(s_{c_m}\circ \cdots\circ s_{c_{2}}\circ s_{c_1})(f_L(e))\in \overline{\sC}.$$ Since $L$ is of zero entropy, we have $$e=(s_{c_m}\circ \cdots\circ s_{c_{2}}\circ s_{c_1})(f_L(e))$$ by Theorem \ref{thm:Og}. Then $$(s_{c_1}\circ s_{c_2} \circ\cdots\circ s_{c_m})(e)=f_L(e),$$ which implies $e=f_L(e)$ since $s_{c_i}(e)=e$ for all $i$. This contradicts $f_S(Ne)\neq Ne.$ \end{proof}

The following basic property of critical sublattices is useful.
\begin{lemma}\label{lem:S}
Let $v\in L_{cr}$. If $r\in L$ is a root with $(v,r)=k$, then $k r\in L_{cr}$.
\end{lemma}

\begin{proof}
Consider the reflection $s_{r}\in W(L)$ defined by the root $r$. Then $$v+k r=v+(v,r)r=s_r(v)\in L_{cr}.$$ Thus, $k r=s_r(v)-v\in L_{cr}.$
\end{proof}

A {\it reducible fiber component} of $e$ is a $\sC$-irreducible root $c$ with $(e,c)=0$. A {\it reducible fiber} of $e$ is a collection of reducible fiber components of $e$ such that the corresponding (extended) Dynkin graph is of type $\widetilde{A}_{i}$ ($i\ge 1$), $\widetilde{D}_{j}$ ($j\ge 4$), or $\widetilde{E}_{k}$ ($k=6,7,8$) (cf. \cite{Ni81}). The roots contained in a reducible fiber are called its {\it irreducible components}.

\begin{theorem}\label{thm:S=L}
Suppose $e$ has a $\sC$-section $d$. Let $M:=\langle e, d\rangle^{\perp}_L$.  Then $L=L_{cr}$ if one of the following conditions is true:
\begin{enumerate}
\item $e$ has a reducible fiber with at least three irreducible components;

\item The genus of $M$ contains at least two classes.
\end{enumerate}
\end{theorem}
\begin{proof}
Suppose (1) is true. By assumption, $e$ has a reducible fiber $\{c_1,...,c_m\}, m\ge 3$. Since $d$ is a $\sC$-section of $e$, we may assume $(d,c_1)=1$. By Theorem \ref{thm:S} and Lemma \ref{lem:S}, both $d$ and $c_1$ are contained in $L_{cr}$. By the intersection numbers $(c_i,c_j)$ and by Lemma \ref{lem:S} again, we have $c_i\in L_{cr}$ for all $i$, which implies $e\in L_{cr}$ since $e$ is a linear combination of $c_i$. Let $v\in M$. Consider the element $x:=-\frac{v^2}{2} e+d+v\in L.$ Then $x^2=-2$ and $(e,x)=1$. Thus, $x, v\in L_{cr}$. Since $L=\langle e, d\rangle\oplus M$, we have $L=L_{cr}$.

Suppose (2) is true. By assumption, there exists a primitive isotropic element $e_1 \in L\cap (\overline{\sC}\setminus\{0\})$ with  finite ${\rm Aut}(\sC)_{e_1}$. Since $(e,e_1)>0$, it follows that for any reducible fiber of $e_1$, at least one irreducible component, say $c$, has positive intersection number with $e$, which implies $c\in L_{cr}$. Thus, by similar arguments as in (1), $L=L_{cr}$ if $e_1$ has a reducible fiber with at least $3$ irreducible components. Then, we are reduced to consider the following case: all reducible fibers of $e$ and $e_1$ have exactly $2$ irreducible components. Since the automorphism group ${\rm Aut}(\sC)_e$ (resp. ${\rm Aut}(\sC)_{e_1}$) is infinite (resp. finite), it follows that the number of reducible fibers of $e$ is strictly less than that of $e_1$. Then we infer that there exists a reducible fiber of $e_1$ such that the two irreducible components are contained in $S$. From this, we have $L=L_{cr}$ by similar arguments as in (1) again. \end{proof}

By slightly adapting the proof of Theorem \ref{thm:S=L}, we have the following result which is very effective for computing critical sublattices in practice.
\begin{theorem}\label{thm:SinL}
Suppose $L$ contains $U$. Then the following statements are true:
\begin{enumerate}
\item If $e$ has no reducible fiber, then $L_{cr}$ is the sublattice generated by all roots in $L$.

\item If $e$ has a reducible fiber, then the index $[L: L_{cr}]$ is at most $2$.

\item If $L_1\subset L$ is a sublattice containing both $e$ and $L_{cr}$, then $L_1=L$. 
\end{enumerate}
In particular, the quotient group $L/L_{cr}$ is a finite cyclic group generated by $e$.
\end{theorem}

The next theorem shows an explicit relation between non-2-reflective zero entropy lattices and $2$-reflective lattices.
\begin{theorem}\label{thm:overlattice}
Suppose $L$ contains $U$. If $L_1$ is a nontrivial even overlattice of $L$, then $L_1$ is $2$-reflective. \end{theorem}

\begin{proof}
Suppose ${\rm O}^+(L_1)/W(L_1)$ is infinite. Then there exists a chamber $\sC_1$ and $e_1\in L$ such that $(L_1, \sC_1, e_1)$ is a zero entropy triple. Without loss of generality, we may assume that there exists $h\in L$ such that $h$ is contained in both $\sC$ and $\sC_1$. Note that $e$ has $\sC$-sections. If $e_1\notin L$, then for some integer $N>1$, $Ne_1$ is a primitive isotropic element in $L\cap(\overline{\sC}\setminus\{0\})$ with infinite ${\rm Aut}(Y)_{Ne_1}$ and no section, a contradiction (Theorem \ref{thm:Og}). So we have $e_1\in L$. Since $L$ contains both $e_1$ and the critical sublattice of $L_1$, we have $L=L_1$ by Theorem \ref{thm:SinL}. Thus, $L_1$ is a trivial overlattice of $L$, a contradiction. \end{proof}

The following lemma shows that Theorem \ref{thm:overlattice} can be applied to many lattices.
\begin{lemma}\label{lem:existoverlattice}
Let $L_1$ be an even hyperbolic lattice and let $p$ be a prime number. Then $L_1$ has an even overlattice $L_2$ with $[L_2:L_1]=p$ if one of the following two statements holds:

\begin{enumerate}
\item $p$ is odd and $p^3$ divides ${\rm det}(L_1)$;

\item $p=2$ and $16$ divides ${\rm det}(L_1)$.
\end{enumerate}
\end{lemma}
\begin{proof}
We sketch the proof. For (1) (resp. (2)), there exists a subgroup $H\subset G(L_1)$ such that $|H|=p^3$ (resp. $|H|=16$). Then a direct calculation shows that there exists an order $p$ isotropic element $x\in H$ such that $q_{L_1}(x)=0$ (cf. \cite[Propositions 1.8.1, 1.8.2]{Ni80}, \cite[Corollary 2.53]{Ge08}). This implies that there is an even overlattice $L_2$ of $L_1$ with $[L_2 : L_1]=p$.
\end{proof}

\begin{theorem}\label{thm:atleast13}
Let $L_1$ be an even hyperbolic lattice of rank at least $13$. If $p^3$ divides ${\rm det}(L_1)$ for some odd prime number $p$, then $L_1$ is of positive entropy.
\end{theorem}

\begin{proof}
Suppose that $L_1$ is of zero entropy. By \cite{Ni81}, the determinants of the $2$-reflective even hyperbolic lattices with rank at least $13$ are powers of $2$.  Thus, $L_1$ is not $2$-reflective. By Lemma \ref{lem:overlattice}, there exists a non-$2$-reflective even overlattice $L_2$ of $L_1$ such that $L_2$  is of zero entropy and contains $U$. Then $L_2=U\oplus M$, where $M$ is an even negative definite lattice. Since ${\rm rk}(M)\ge 11$, the genus of $M$ is not a one-class genus by \cite[Theorem 1]{Wa63}. Hence, by Theorem \ref{thm:S=L}, we have $L_2=(L_2)_{cr}=L_1$. Then by $p^3\;| \; {\rm det}(L_1)$ and Lemma \ref{lem:existoverlattice}, we have that $L_1$ has an even overlattice $L_3$ with $[L_3:L_1]=p$. Then $p\;|\; {\rm det}(L_3)=\frac{{\rm det}(L_1)}{p^2}.$
Hence, by \cite{Ni81} again, $L_3$ is not $2$-reflective. However, $L_3$ is $2$-reflective by Theorem \ref{thm:overlattice}, a contradiction.
\end{proof}

\begin{remark}\label{rem:2-elem}
Let $L_1$ be a $2$-elementary even hyperbolic lattice with $12\le {\rm rk}(L_1)\le 18$ and ${\rm rk}(L_1)+l(L_1)=22$. In \cite{OY19}, it is proved that $L_1$ is of positive entropy by geometric arguments using elliptic fibrations. Here we give a different proof based on Theorem \ref{thm:overlattice}. Note that $L_1$ contains $U$. By $l(L_1)\ge 4$ and Lemma \ref{lem:existoverlattice}, $L_1$ has an even overlattice $L_1^\prime$ with $[L_1^\prime :L_1]=2$. Thus, $L_1^\prime$ is $2$-elementary and ${\rm rk}(L_1^\prime)+l(L_1^\prime)=20$. Hence, $L_1^\prime$ is not $2$-reflective by \cite{Ni81}. Then by Theorem \ref{thm:overlattice}, $L_1$ is of positive entropy. 
\end{remark}

The $18$ rank $3$ lattices in $\sI$ were obtained in \cite[Section 3.3]{Ni99} based on Vinberg's algorithm (for the $8$ lattices containing $U$, see \cite{Me21} for a different approach).  So the results in the following example is not new. However, the example provides an illustration for computing critical sublattices.
\begin{example}\label{exm:rk3}
The set of all rank $3$ non-2-reflective zero entropy lattices containing $U$ is $\{L_{k}:=U\oplus [-2k]| \; k=2, 3, 4, 5, 7, 9, 13, 25\}$. Let $L:=L_2$. Choose a basis $e_1, e_2, v$ of $L$ with $e_1^2=e_2^2=(e_1,v)=(e_2,v)=0,  (e_1,e_2)=1, v^2=-4$. By Theorem \ref{thm:SinL} (1), $L_{cr}$ is generated by all roots in $L$. Let $r:=(x,y,z)\in L$, where $x,y,z$ are coordinates with respect to the basis $e_1, e_2, v$. If $r^2=-2$, then $xy-2z^2=-1$, and $x, y$ are odd integers. Consider the roots $r_1:=(1,-1,0)$, $r_2:=(1,1,1)$, $r_3:=(1,1,-1)$. Since $4e_1=2r_1+r_2+r_3$,  by Theorem \ref{thm:SinL}, $[L:L_{cr}]=1, 2$, or $4$. On the other hand, consider the group homomorphism 
$$\varphi: L\longrightarrow \Z/4\Z\oplus \Z/4\Z\oplus \Z/2\Z,\;\; \varphi ((x,y,z))=(\overline{x}, \overline{y}, \overline{z}).$$ 
If $r^2=-2$, then  $\varphi(r)\in H:=\langle (\overline{1},\overline{-1},\overline{0}), (\overline{1},\overline{1},\overline{1}) \rangle$. Clearly $\varphi (2e_1)\notin H$. Thus, $2e_1\notin L_{cr}$. Then $[L:L_{cr}]=4$, $L_{cr}=\langle r_1, r_2, r_3\rangle$, and $L$ has exactly three  rank $3$ zero entropy sublattices: $L_{cr}$, $\langle r_1, r_2, 2e_1\rangle$, $L$. Similarly, we obtain all rank $3$ zero entropy sublattices of $L_k$, and $[L_k:(L_k)_{cr}]=4, 6, 8, 1, 2, 3, 1, 1$ for $k=2, 3, 4, 5, 7, 9, 13, 25$ respectively.\end{example}


\section{Tests for entropy-positivity of lattices}\label{sec:tests}
In general it is hard to determine that a lattice is of zero/positive entropy because the number of roots in it may be infinite. In this section, we collect some tests for entropy-positivity of lattices which will be used in the next section.

Let $L:=U\oplus M$, where $M$ is an even negative definite lattice of rank $m \ge 2$. In practice, we use combination of the following tests to check entropy-positivity of $L$.

\subsection*{Overlattice test} The overlattice test is based on Theorem \ref{thm:overlattice}. Let $\sF^{m+2}$ be the finite list of all $2$-reflective even hyperbolic lattices of rank $m+2$. We compute the finite set $\sM$ which consists of the isometry classes of even overlattices $M^\prime \supset M$ with prime index $[M^\prime: M]$. If there exists $M^\prime \in \sM$ such that $U\oplus M^\prime \notin\sF^{m+2}$ (it suffices to compare the genus of $U\oplus M^\prime$ with the genera in $\sF^{m+2}$ using Sage \cite{The} since they are one-class genera), then output that $L$ is of positive entropy and stop. If there is no such $M^\prime$, then output that unclear and stop.

\subsection*{Sublattice test} The sublattice test is based on the following theorem which is a generalization of \cite[Theorem 3.6]{Me21}. 

\begin{theorem}\label{thm:subtest}
Let $L$ be an even hyperbolic lattice of rank $n$ and let $L_1\subset L$ be a primitive hyperbolic sublattice of positive entropy. Then $L$ is of positive entropy if one of the following two statements holds:
\begin{enumerate}
\item  ${\rm rk}(L_1)=n-1$ and $|\frac{{\rm det}(L)}{{\rm det}(L_1)}|\ge 2$;

\item $L=L_1\oplus (L_1)_L^\perp$.
\end{enumerate}
\end{theorem}
\begin{proof}
Suppose (1) is true. Let $N:=(L_1)_L^\perp$. By assumption, there exists $f\in {\rm Aut}(\sC_1)$ of positive entropy, where $\sC_1$ is a chamber of $\sP_{L_1}$. Since $G(L_1)$ is a finite group, replacing $f$ by its sufficiently large power if necessary, we may assume that $\overline{f}| G(L_1)={\rm id}_{G(L_1)}$. By ${\rm rk}(N)=1$, we have $N=\langle v\rangle$ with $v^2=-a$ for some positive even integer $a$. Moreover, $L$ can be identified with $L_1\oplus_{\phi} N$ for some gluing map $\phi: H_1\longrightarrow H_2$, where $H_1\subset G(L_1)$, $H_2\subset G(N)$ (cf. \cite[Section 4]{OY20}). Let $k=|H_1|=|H_2|$. Then we have 
$$L=\{(\frac{x}{k},\frac{mv}{k}) |\; m\in \mathbb{Z}, x\in L_1, \overline{\frac{x}{k}}\in H_1 ,\overline{\frac{mv}{k}}\in H_2, \phi (\overline{\frac{x}{k}})=\overline{\frac{mv}{k}}\}.$$
Suppose $(\frac{x}{k}, \frac{mv}{k})\in L$ is a root which is not in $L_1\oplus N$. Note that 
$$-2\ge -|\frac{{\rm det}(L)}{{\rm det}(L_1)}|=-|\frac{{\rm det}(L_1)a}{{\rm det}(L_1)k^2}|=-\frac{a}{k^2}\ge -\frac{m^2 a}{k^2}={(\frac{mv}{k})}^2.$$ 
Then by \cite[Theorem 5.2]{Mc11a}, $$f \oplus_\phi {\rm id}_N\in {\rm O}^+(L_1\oplus_\phi N)={\rm O}^+(L)$$ preserves a chamber.

Case (2) is a direct consequence of \cite[Theorem 5.2]{Mc11a}.\end{proof}

 Now we describe the sublattice test. Let $\sZ^{m+1}$ be the finite list of all zero entropy even hyperbolic lattices of rank $m+1$. For each trial, we (randomly) choose a list of $m+1$ elements in $L$. If the sublattice $L_1$ generated by the $m+1$ elements is a primitive sublattice of $L$ such that ${\rm rk}(L_1)=m+1$, $|\frac{{\rm det}(L)}{{\rm det}(L_1)}|\ge 2$, and $L_1\notin \sZ^{m+1}$, then output $L$ is of positive entropy and stop. If such $L_1$ is not found after a reasonable finite number of trials, then output unclear and stop. 

\subsection*{Genus test} It is known that $L=U\oplus M$ is of positive entropy if the genus of $M$ contains at least two classes with no root sublattices of finite index. For relatively small rank/determinant, we may use Magma \cite{BCP} to compute all classes quickly. The issue is how to compute the classes when the rank/determinant is relatively big (see \cite{Me21} for another approach). Our idea is to collect (some of) the classes via random search, which turns out to be effective at least during our classification in Section 5. We (randomly) choose many $v\in M$. For each $v$, we pick an element $x_v\in U$ with $x_v^2=-v^2$. Let $e_v:=x_v+v\in L$. If there is $y_v\in L$ with $(e_v, y_v)=1$, then we compute $N_v:= \langle e_v, y_v \rangle^{\perp}_{L}$. If we find $v_i$ ($i=1, 2$) with both $N_{v_1}\ncong N_{v_2}$ having no root sublattices of finite index, then output that $L$ is of positive entropy and stop. Otherwise, output unclear and stop.  Note that the mass formula (\cite{CS88}) can be used to confirm if all classes have been found.

\subsection*{Surjectivity test} If $M$ has no root sublattices of finite index and the natural homomorphism ${\rm O}(M)\longrightarrow {\rm O}(q_M)$ is not surjective, then $L=U\oplus M$ is of positive entropy (\cite[Theorem 3.12]{Me21}). For many $M$ in our study,  we use Magma to compute ${\rm O}(M)$ (generators, order, etc.) and then we can check the surjectivity.

\subsection*{Covering radius test}  The {\it covering radius} of a positive definite lattice $N$ is the smallest positive real number $R$ for which spheres of radius $R$ centered at the elements in $N$ cover all the space $N\otimes \R$ (cf. \cite{CS99}). The covering radius of the Leech lattice $\Lambda $ is $\sqrt{2}$, which plays a key role in Conway's characterization of the automorphism group of $U\oplus \Lambda$ based on Vinberg's algorithm (\cite{Vi75}). A slight variant of the proof shows the following

\begin{theorem}[cf. {\cite{Co83}}, {\cite{Bo90}}]\label{thm:cradius}
Let $M$ be an even negative definite lattice and $L:=U\oplus M$. Let $e, e^\prime$ be a basis of $U$ with $e^2={e^\prime}^2=0, (e,e^\prime)=1$.  Suppose the covering radius of $M(-1)$ is at most $\sqrt{2}$. Then there is a chamber  $\sC\subset \sP_L$ such that $e\in L\cap (\overline{\sC}\setminus \{0\})$ and $(c,e)\leq 1$ for all $\sC$-irreducible roots $c\in L$. In particular, $L$ is of zero entropy.\end{theorem}

For many $L=U\oplus M$ listed in the Appendix, we can compute the covering radius of $M(-1)$ by Magma to confirm that $L$ is of zero entropy. 

\subsection*{Shimada's algorithm} In \cite{Sh15} Shimada presented an algorithm which is applicable to compute the automorphism groups of a wide class of K3 surfaces by generalizing Borcherds' method (\cite{Bo87}, \cite{Bo98}). We briefly explain how we implement Shimada's algorithm. Firstly, we find a primitive embedding $L\hookrightarrow U\oplus E_8^m$, where $m=1,2$. Consider the subgroup $$G:=\{g\; |\; f\in W(U\oplus E_8^m),\, f(L)=L,\, g=f|L\}\subset {\rm O}^+(L)$$ which satisfies $[{\rm O}^+(L): G]<+ \infty$ and the required properties to proceed \cite[Algorithm 6.1]{Sh15}. Then we run the algorithm to compute a finite set ${g_1,...,g_k}$ of generators of the group $H:=\{g\in G |\; g (\sC)=\sC \}$, where $\sC\subset \sP_L$ is a chamber with respect to the Weyl group $W(L)$ action on $\sP_L$. Finally, we compute the sublattice $L_1:=\{x\in L |\; g_i(x)=x \text{ for all }i\}\subset L$. If $L_1$ is negative definite, then $L$ is of positive entropy; otherwise, $L$ is of zero entropy. 

\begin{remark}\label{rem:tests}
(i) Shimada's algorithm is powerful since it gives a definite answer for entropy-positivity of lattices if the calculation can be finished. The cost to pay is that for lattices of relatively big determinant, the computation involved might be too heavy (see \cite[Remark 10.3]{Sh15}). In general, such lattices are more likely of positive entropy. Hence, in practice, to check entropy-positivity of many candidates of the form $L=U\oplus M$, we try the first four tests above before applying Shimada's algorithm, and they frequently pick out all the candidates of positive entropy quickly (see Remark \ref{rem:exception}). 

(ii) Note that the sublattice test and Shimada's algorithm are applicable to even hyperbolic lattices not necessarily containing $U$.
\end{remark}


\section{Proof of Theorem \ref{main1}}\label{sec:proof}
\noindent

In this section, we prove Theorem \ref{main1}. The Picard number $\rho=\rho(X)$ of any complex projective K3 surface $X$ is at most $20$. It is known that $\sI^{\rho}$ is empty if $\rho =20, 19$ (\cite[Theorem 1.6]{Og07}, \cite[Theorem 0.3]{Me21}). If $\rho\le 2$, then $\sI^{\rho}$ is empty (see e.g. \cite[Chapter 15, Example 2.11]{Hu16}). So we only need to determine $\sI^\rho$, $3\le \rho\le 18$. Since ${\rm NS}(X)$ is a primitive sublattice of the unimodular rank $22$ lattice $U^3\oplus E_8^2$, it follows that $\rho(X)+l({\rm NS}(X))\le 22$. Thus, it suffices to classify all lattices $L$ satisfying the following conditions:
\begin{enumerate}
\item $L$ is an even hyperbolic non-$2$-reflective lattice;

\item $L$ is of zero entropy;

\item $3\le {\rm rk}(L)\le 18$ and ${\rm rk}(L)+l(L)\le 22$.
\end{enumerate}
 We classify such lattices inductively from rank $3$ to $18$. For each rank $n$, by Lemma \ref{lem:overlattice} and Theorem \ref{thm:S}, we can obtain all lattices satisfying (1)-(3) in two steps: 
\begin{enumerate}
\item[(I)]:  classify rank $n$ lattices $L$ satisfying (1)-(3) of the form $L=U\oplus M$;

\item[(II)]: compute the criticral sublattices $L_{cr}$ and their overlattices inside $ L$. 
\end{enumerate}
 If $L=U\oplus M$ satisfies (1)-(3) and $M$ has no roots, then  $M$ is isometric to one of the $32$ lattices in \cite[Theorem 5.12]{Me21}. It is known that two even lattices $M_i$ ($i=1,2$) have the same genus if and only if $U\oplus M_1$ is isometric to $U\oplus M_2$. Thus, in Step (I), we are reduced to consider genera of even negative definite lattices $M$ belonging to one of the following two types: 
 \begin{enumerate}
\item[(i)] $M$ is an overlattice of a root lattice;
 
\item[(ii)] The genus of $M$ contains only $M$, $M$ has roots, and $M$ has no root sublattice of finite index.
\end{enumerate}

An explicit finite list of all negative definite primitive one-class genera in rank $\ge 3$ is known by a result of Watson, recently corrected by Lorch-Kirschmer \cite{LK13} (for rank $2$, a finite list is also known but a proof for its completeness without assuming the Generalized Riemann Hypothesis seems unknown, cf. \cite{Vo07}). The largest rank in Watson's list is $10$. For ${\rm rk}(M)\ge 3$,  all type (ii) lattices can be derived from Watson's list as follows: firstly, for odd lattices $M^\prime$ in the list, we replace them by $M^\prime(2)$; secondly, we collect all lattices in this modified list which have roots but contain no root sublattice of finite index. It is known $\sI^3$ consists of $18$ lattices (\cite[Section 3.3]{Ni99}, see also Example \ref{exm:rk3}). For computation in cases $4\le n\le 18$, we use McMullen's package \cite{Mc11b} and a mixture of Mathematica (\cite{Wo}), Magma (\cite{BCP}), PARI/GP (\cite{Th}), and SageMath (\cite{The}). Next we explain two typical cases $n=4,12$. 

\subsection*{Case $n=4$} There are two rank $2$ root lattices $A_1^2$, $A_2$ and they do not have nontrivial even overlattices. The lattices $U\oplus A_1^2$, $U\oplus A_2$ are $2$-reflective. Thus, there is no desired lattice $L$ from $M$ of type (i). For type (ii) lattices, we first need to obtain a finite list of candidates. Let $L=U\oplus M$ and let $e_1, e_2$ (resp $\alpha, \beta$) be a basis of $U$ (resp. $M$)  with Gram matrix $\small{\left(\begin{array}{cc} 
0 & 1 \\
1 & 0 
\end{array} \right)}$ (resp.  case (1): $\small{\left(\begin{array}{cc} 
-2 & 0 \\
0 & -2a 
\end{array} \right)}$ or case (2): $\small{\left(\begin{array}{cc} 
-2 & 1 \\
1 & -2a 
\end{array} \right)}$), where $a\ge 2$. 

Case (1): By the sublattice test, we may assume $a=2,3,4,5,7,9,13,25$. So we get $8$ candidates.

Case (2): Let $L_1:=\langle e_1, \alpha, 11 e_2-\beta  \rangle.$ Clearly $L_1$ is a primitive rank $3$ sublattice of $L$ and ${\rm det}(L_1)=242$. Note that there is no rank $3$ lattices in $\sF\cup \sI^3$ of determinant $242$. Thus, $L_1$ is of positive entropy. Therefore, if $|{\rm det}(L)|=4a-1\ge 484,$ then $L$ is of positive entropy by the sublattice test. Hence, we may assume $2\le a\le 121$. So we get $120$ candidates.

 Totally, we have $128$ candidates from cases (1) and (2). After applying the first four tests (overlattice test, sublattice test, genus test, surjectivity test) in Section \ref{sec:tests}, $120$ of them are ruled out and the $8$ candidates $M$ remaining are as follows:

$$\small{\left(\begin{array}{cc} 
-2 & 1 \\
1 & -4 
\end{array} \right)}
, 
\small{\left(\begin{array}{cc} 
-2 & 0 \\
0 & -4 
\end{array} \right)}
, 
\small{\left(\begin{array}{cc} 
-2 & 1 \\
1 & -6 
\end{array} \right)}
, \small{\left(\begin{array}{cc} 
-2 & 0 \\
0 & -6 
\end{array} \right)}
, \small{\left(\begin{array}{cc} 
-2 & 0 \\
0 & -8
\end{array} \right)}
, \small{\left(\begin{array}{cc} 
-2 & 1 \\
1 & -10 
\end{array} \right)}
,$$ 
$$\small{\left(\begin{array}{cc} 
-2 & 1 \\
1 & -14 
\end{array} \right)}
, \small{\left(\begin{array}{cc} 
-2 & 1 \\
1 & -22 
\end{array} \right)}.
$$ 

The covering radii of the first four (resp. the last four) $M(-1)$ are at most (resp. greater than) $\sqrt{2}$. Thus, the first four candidates $L=U\oplus M$ are of zero entropy for sure (Theorem \ref{thm:cradius}). Then by applying Shimada's algorithm to the $4$ candidates remaining, only the last one is of positive entropy. Therefore, by adding the $10$ rank two lattices in \cite[Theorem 5.12]{Me21},  we obtain exactly $17$ desired lattices in step (I). Similar to Example \ref{exm:rk3}, in step (II), we obtain $7$ desired lattices by computing the critical sublattices and their overlattices based on Theorem \ref{thm:SinL}. Therefore, there are exactly $24$ rank four lattices satisfying (1)-(3).

\subsection*{Case $n=12$} It is known that $U\oplus A_1^6\oplus D_4$ is the only $2$-elementary rank $12$ lattices satisfying (1)-(3) (\cite{Ni81}, \cite{Ni20}, see also \cite{OY19}, Remark \ref{rem:2-elem}). Next we only need to treat non-2-elementary lattices in step (I). After applying the first four tests in Section \ref{sec:tests} to the genera of non-$2$-elementary lattices $M$ of type (i), we get $4$ candidates remaining:  $$U\oplus A_1\oplus D_9, U\oplus A_2\oplus D_8, U\oplus A_3\oplus D_7, U\oplus A_2^2\oplus E_6.$$ By Shimada's algorithm, all of them are of zero entropy. From Watson's list, we get only $1$ candidate $L=U\oplus M$ with $M$ of type (ii): $ U\oplus E_8(3)\oplus A_2,$ which is ruled out by Lemma \ref{lem:existoverlattice} and Theorem \ref{thm:overlattice} since the determinants of the $4$ rank $12$ lattices in $\sF$ are $3,4,16,64$. So we totally obtain $5$ desired lattices in step (I). In step (II), by Theorem \ref{thm:SinL} (1), the critical sublattices $L_{cr}=L$ for the $5$ lattices $L$. Thus, there are exactly $5$ rank twelve lattices satisfying (1)-(3).

The other cases (i.e., $n=5,6,7,8,9,10,11,13,14,15,16,17,18$) can be done similarly. Note that for $n\ge 13$, many candidates can be ruled out by Theorem \ref{thm:atleast13}. Totally, there are exactly $193$ lattices, listed in the Appendix, satisfying the conditions (1)-(3). All of them have primitive embeddings into the K3 lattice $U^3\oplus E_8^2$ by \cite[Corollary 1.12.3]{Ni80}. This completes the proof of Theorem \ref{main1}.

\begin{remark}\label{rem:exception}
Among all the non-2-elementary candidates $L=U\oplus M$ with $M$ of type either (i) or (ii) above, almost all the lattices $L$ which are of positive entropy can be ruled out by applying the first four tests in Section 4. In fact, it turns out that the only exception for which all of them fail is the lattice $U\oplus \small{\left(\begin{array}{cc} 
-2 & 1 \\
1 & -22 
\end{array} \right)}$. It would be interesting if there is an explanation.
\end{remark}

\appendix

\section{The $193$ lattices}

 In this appendix, we list the $193$ lattices isometric to the N\'eron-Severi lattices of the projective K3 surfaces of zero entropy with infinite automorphism groups.  For a symmetric matrix $(a_{ij})$, we use its lower left entries $[a_{11},a_{21},a_{22},...,a_{nn}]$ to denote the lattice with Gram matrix equal to it. For example, by $[0,2,2,0,1,-14]$ we mean the rank $3$ lattice with Gram matrix $\small{\left(\begin{array}{ccc} 
0 & 2 & 0\\
2& 2 & 1\\
0& 1& -14
\end{array} \right)}$.  
The rank of the $193$ lattices is denoted by $n$, $3\le n\le 18$. The list is as follows (Gram matrices, determinants and glue groups of these lattices can be found in the ancillary file).

\subsection*{$n=3$ ($18$ lattices):} $${\displaystyle U\oplus A_1(2)},\,\,\, {\displaystyle U\oplus A_1(3)},\,\, {\displaystyle U\oplus A_1(4)},\,\, {\displaystyle U\oplus A_1(5)},\,\,\,{\displaystyle U\oplus A_1(7)},\,\,\, {\displaystyle [2]\oplus A_1\oplus A_1(2)}, \,\,\,{\displaystyle U\oplus A_1(9)},$$ $${\displaystyle [8]\oplus A_2}, \,{\displaystyle U\oplus A_1(13)},\, {\displaystyle [2]\oplus A_1\oplus A_1(4)},\, {\displaystyle U\oplus A_1(25)},\, {\displaystyle [18]\oplus A_2},\, {\displaystyle [0,2,2,0,1,-14]},\,\,{\displaystyle [16]\oplus A_1^2},$$ $${\displaystyle [8]\oplus A_1\oplus A_1(4)},\,\, {\displaystyle [0,3,4]\oplus A_1(9)},\,\, {\displaystyle [72]\oplus A_2},\,\, {\displaystyle U(16)\oplus A_1}.$$

\subsection*{$n=4$ ($24$ lattices):} $$ U\oplus [-2,1,-4], U\oplus A_1\oplus A_1(2), U\oplus [-2,1,-6], U\oplus A_1\oplus A_1(3), U\oplus A_2(2), U\oplus [-4,1,-4],$$ $${\displaystyle  U\oplus A_1\oplus A_1(4)},\, {\displaystyle U\oplus A_1^2(2),}\, {\displaystyle U\oplus [-2,1,-10]},\,{\displaystyle U\oplus [-6,2,-4]},\,{\displaystyle U\oplus A_1(2)\oplus A_1(3)},\,{\displaystyle U\oplus [-2},$$ $$1,-14], {\displaystyle U\oplus A_2(3)}, {\displaystyle [8]\oplus A_3},{\displaystyle U\oplus [-6,1,-6]}, {\displaystyle U\oplus [-10,2,-4]}, {\displaystyle U\oplus A_1(5)\oplus A_1(2)},  {\displaystyle U(4)\oplus A_2},$$ $${\displaystyle  [2]\oplus A_1\oplus A_2(2)},\,{\displaystyle  [2]\oplus A_1\oplus A_1^2(2)},\, {\displaystyle  U\oplus [-14,3,-6],}\, {\displaystyle [20]\oplus A_3,\, [0,3,4]\oplus A_2(3),\, U(8)\oplus A_1^2.}$$

\subsection*{$n=5$ ($27$ lattices):} $${\displaystyle [2]\oplus A_4},\,\, {\displaystyle U\oplus [-4]\oplus A_2},\,\,  U\oplus [-2,0,-2,1,-1,-4],\,\, {\displaystyle U\oplus A_1\oplus [-2,1,-4]},\,\, U\oplus [-2,-1,-2,1,$$ $$1,-6],\,\,{\displaystyle U\oplus A_1^2\oplus [-4]},\,\, {\displaystyle [4]\oplus A_4},\,\,{\displaystyle [2,1,-2]\oplus A_3},\,\,{\displaystyle U\oplus [-2,-1,-4,-1,0,-4]},\,\,{\displaystyle U\oplus A_2\oplus [-8]},$$ $${\displaystyle U\oplus A_1\oplus A_2(2)},\,\, {\displaystyle U\oplus A_1^2\oplus A_1(3)},\,\,{\displaystyle U\oplus [-2,-1,-4,0,-1,-4]},\,\, {\displaystyle U\oplus A_1\oplus A_1^2(2)},\,\,{\displaystyle U\oplus A_3(2)},$$ $${\displaystyle [2,1,-4]\oplus A_3},\,{\displaystyle [4,1,-2]\oplus A_3},\, {\displaystyle U\oplus [-4,1,-4,2,1,-4]},\, {\displaystyle U\oplus A_1(2)\oplus [-2,-1,-6]},\, {\displaystyle [10]\oplus A_4},$$ $${\displaystyle U\oplus [-4,1,-4,1,-1,-4]},\,\, {\displaystyle [14]\oplus D_4},\,\, {\displaystyle U(4)\oplus A_3},\,\, {\displaystyle U\oplus [-6]\oplus A_2(2)},\,\,{\displaystyle [0,2,2,0,1,-2]\oplus A_2(2)},$$ $${\displaystyle U(5)\oplus A_3},\,\, {\displaystyle [2]\oplus A_1\oplus A_3(2)}.$$

\subsection*{$n=6$ ($28$ lattices):} $${\displaystyle [2]\oplus A_5},\,\, {\displaystyle U\oplus [-2,-1,-2,-1,0,-2,0,-1,0,-4]},\,\,{\displaystyle  U\oplus A_3\oplus [-4]},\,\,{\displaystyle U\oplus [-2,-1,-2,0,0, -2},$$ $$-1,0,-1,-4],\,\, {\displaystyle U(2)\oplus A_4},\,\,{\displaystyle  [2,0,-2]\oplus A_4},\,\,{\displaystyle [2,1,-2]\oplus D_4},\,\, {\displaystyle U\oplus A_2\oplus [-2,1,-4]}, \,\,U\oplus A_1 \oplus$$ $$[-2,0,-2,1,1,-4],\,\, {\displaystyle [2,1,-2]\oplus A_4},\,\,{\displaystyle  U\oplus [-2,-1,-2,-1,0,-4,0,-1,-1,-4]},\,\, U\oplus [-2,-1,$$ $$-2,1,1,-4,1,0,-1,-4],\,\,{\displaystyle  [-2, 0, -8, 1, 4, -2]\oplus A_3},\,\, {\displaystyle U\oplus A_1^3\oplus A_1(2)},\,\, U\oplus [-2,0, -2,-1,-1,$$ $$-4,-1,0,-2,-4],\,\,{\displaystyle  U(3)\oplus D_4},\,\, {\displaystyle U(2)\oplus A_2^2},\,\,{\displaystyle  [2,1,-4]\oplus D_4},\,\, {\displaystyle [4,1,-2]\oplus A_4},\,\,{\displaystyle  [4]\oplus A_2\oplus A_3,}$$ $${\displaystyle U\oplus A_1^2(2)\oplus A_2},\,\,{\displaystyle  [10]\oplus A_5},\,\, {\displaystyle U\oplus A_1\oplus A_3(2)},\,\,{\displaystyle  U\oplus D_4(2)}, \,\, {\displaystyle [4,1,-2]\oplus A_2^2},\,\, {\displaystyle U(3)\oplus A_1^2\oplus A_2},$$ $${\displaystyle  U(5)\oplus A_4},\,\,{\displaystyle  [2]\oplus A_1\oplus D_4(2).}$$

\subsection*{$n=7$ ($21$ lattices):} $$[4]\oplus E_6,\,\, [2]\oplus A_6,\,\, U\oplus [-2,-1,-2,-1,-1,-2,1,1,1,-2,1,1,1,-1,-4], U\oplus A_1(2)\oplus D_4,$$ $$[2,1,-2]\oplus D_5,\,\, U\oplus [-2,0,-2,0,0,-2,-1,-1,0,-2,0,-1,-1,0,-4],\,\, U\oplus A_1^3\oplus A_2,\,\, U(2)\oplus A_5,$$ $$U\oplus [-2,-1,-2,0,0,-2,0,0,-1,-2,1,1,1,0,-4],\,\, [4]\oplus A_6,\,\, [10]\oplus E_6,\,\, [-2, 0, -8, 1, 4,-2]\oplus$$ $$D_4,\,\,\, [8]\oplus D_6,\,\,\, U(3)\oplus D_5,\,\,\,[-2, 0, -8, 1, 4, -2]\oplus A_4,\,\, \,U(2)\oplus A_2\oplus A_3,\,\,\, [18]\oplus E_6,\,\,\, [4]\oplus A_3^2,$$ $$[16]\oplus D_6,\,\, U\oplus A_1\oplus D_4(2),\,\, U(3)\oplus A_1\oplus A_2^2.$$

\subsection*{$n=8$ ($19$ lattices):} $${\displaystyle U\oplus [-2,-1,-2,-1,-1,-2,-1,0,0,-2,-1,-1,0,-1,-2,-1, 0,-1,0,0,-4]},\,\,{\displaystyle U(2)\oplus E_6},$$ $${\displaystyle [2,1,-2]\oplus E_6},\,\, {\displaystyle U\oplus [-2,1,-2,-1,1,-2,-1,1,-1,-2,1,-1,1,1,-2,1, 0,1,1,-1, -4]}, \,\, U\oplus$$ $$A_1^2\oplus A_4, \,\,{\displaystyle [10]\oplus E_7},\,\, {\displaystyle U\oplus [-2,0,-2,0,-1,-2,0,-1,0,-2,0,1,0,1,-2,-1,1,1, 1,-1,-4]},$$ $${\displaystyle U\oplus A_1\oplus A_2\oplus A_3},\,\, {\displaystyle [4,1,-2]\oplus E_6},\,\, {\displaystyle U(2)\oplus A_6},\,\, {\displaystyle U\oplus A_1^3\oplus A_3},\,\, {\displaystyle [4]\oplus A_7},\,\,{\displaystyle  [16]\oplus E_7}, U\oplus A_1^2\oplus$$ $$A_2^2,\, {\displaystyle  U\oplus [-2,1,-2,0,0,-2,0,0,1,-2,-1,1,-1,1,-4,-1,0,-1,0,-2,-4]},\, {\displaystyle U(2)\oplus A_2\oplus D_4},$$ $${\displaystyle  U(2)\oplus A_3^2},\,\,{\displaystyle  U(4)\oplus D_6},\,\,{\displaystyle U(3)\oplus A_2^3}.$$

\subsection*{$n=9$ ($15$ lattices):} $$U\oplus [-2,1,-2,1,-1,-2,1,-1,-1,-2,1,-1,0,-1,-2, 1,-1,0,0, -1,-2,1,-1, 0,-1,-1,$$ $$-1,-4],\,\, [10]\oplus E_8,\,\, U\oplus A_1\oplus A_6,\,\, U\oplus A_1^2\oplus D_5,\,\, [16]\oplus E_8,\,\, U\oplus A_2\oplus A_5,\,\, [4,1,-2]\oplus E_7,$$ $$U\oplus A_3\oplus A_4, U\oplus A_1^2\oplus A_5, [-2, 0, -8, 1, 4, -2]\oplus E_6, U\oplus A_1\oplus A_2\oplus A_4, U(4)\oplus E_7, U(2)\oplus A_7,$$ $$U\oplus A_1\oplus A_2^3,\,\, U(2)\oplus A_3\oplus D_4.$$

\subsection*{$n=10$ ($13$ lattices):} $$[2,1,-2]\oplus E_8,\, U\oplus A_1\oplus D_7,\, U\oplus A_8,\, [4,1,-2]\oplus E_8,\, U\oplus A_1^2\oplus E_6,\, U\oplus A_1\oplus A_7,\, U\oplus A_1(2)\oplus$$ $$D_7,\,\, U\oplus [-2,-1,-2,1,0,-2,1,0,-1,-2,-1,-1,1,1, -2,-1,-1,1,1,-1,-2,-1, 0,0,1,0,$$ $$0,-4,-1,0,0,1,0,0,0,-4],\,\, U(4)\oplus E_8,\,\,U\oplus A_4^2,\,\, U\oplus A_2^4, \,\, U\oplus E_8(2),\,\, U(2)\oplus A_1^8.$$

\subsection*{$n=11$ ($6$ lattices):} $$U\oplus D_9,\,\, U\oplus E_8 \oplus A_1(3),\,\, [-2, 0, -8, 1, 4, -2]\oplus E_8,\,\, U\oplus A_2\oplus D_7,\,\, U\oplus A_2\oplus A_1 \oplus E_6,$$ $$U\oplus A_1 \oplus E_8(2).$$

\subsection*{$n=12$ ($5$ lattices):} $$U\oplus A_1\oplus D_9,\,\, U\oplus A_2 \oplus D_8,\,\, U\oplus A_3 \oplus D_7,\,\, U\oplus A_2^2\oplus E_6,\,\, U\oplus A_1^6 \oplus D_4.$$

\subsection*{$n=13$ ($3$ lattices):} $$U\oplus A_2\oplus A_1 \oplus E_8,\,\, U\oplus D_4\oplus D_7,\,\, U\oplus A_1^3\oplus D_4^2.$$

\subsection*{$n=14$ ($5$ lattices):} $$U\oplus A_4 \oplus E_8,\,\, U\oplus A_1\oplus D_{11},\,\, U\oplus E_6^2,\,\, U\oplus D_4^3,\,\, U\oplus A_1^4\oplus D_8.$$

\subsection*{$n=15$ ($3$ lattices):} $$U\oplus D_{13},\,\, U\oplus A_5\oplus E_8,\,\, U\oplus A_1^3\oplus D_{10}.$$

\subsection*{$n=16$ ($2$ lattices):} $$U\oplus E_6\oplus E_8,\,\, U\oplus A_1^2\oplus D_{12}.$$

\subsection*{$n=17$ ($2$ lattices):} $$U\oplus D_{15},\,\, U\oplus A_1\oplus D_{14}.$$

\subsection*{$n=18$ ($2$ lattices):} $$U\oplus D_{16},\,\, U\oplus A_1 \oplus E_7\oplus E_8.$$


\begin{thebibliography}{BlEsKe14} 

\bibitem[AST11]{AST11} Artebani, M., Sarti, A., Taki, S.: {\em K3 surfaces with non-symplectic automorphisms of prime order}, Math. Z. 268, 507–533 (2011)

\bibitem[BHPV04]{BHPV04} Barth, W. P., Hulek, K.; Peters, C. A. M., Van de Ven, A. : {\em Compact complex surfaces}, Second edition. Ergebnisse der Mathematik und ihrer Grenzgebiete {\bf 4}, Springer-Verlag, Berlin, 2004. 
%
\bibitem[BK14]{BK14} Bedford, E., Kim, K.: {\em Dynamics of (pseudo) automorphisms of 3-space: periodicity versus positive entropy}, Publ. Mat. 58 (2014), no. 1, 65–119.

\bibitem[Bo87]{Bo87}Borcherds, R. E.: {\em Automorphism groups of Lorentzian lattices}, J. Algebra 111, no. 1 (1987): 133–53.

\bibitem[Bo90]{Bo90} Borcherds, R. E.: {\em Lattices like the Leech lattice}, J. Algebra 130 (1990), 219–234.

\bibitem[Bo98]{Bo98} Borcherds, R. E.: {\em Coxeter groups, Lorentzian lattices, and K3 surfaces}, Int. Math. Res. Not. 1998, no. 19 (1998): 1011–31.

\bibitem[BCP]{BCP}Bosma, W., Cannon, J., Playoust, C.: {\em The Magma algebra system I: The user language}, J. Symb. Comp. 24 (1997), 235-265; home page for Magma version 2.23-1 (2017) at http://magma.maths.usyd.edu.au.

\bibitem[BE21]{BE21} Brandhorst, S., Elkies, N. D.:, {\em Equations for a K3 Lehmer map}, https://arxiv.org/pdf/2103.15101.pdf.
%
\bibitem[Ca99]{Ca99} Cantat, S. : {\em Dynamique des automorphismes des surfaces projectives complexes}, C. R. Acad. Sci. Paris S\'er. I Math. {\bf 328} (1999) 901--906.

\bibitem[Co83]{Co83} Conway, J. H.: {\em The automorphism group of the 26-dimensional even unimodular Lorentzian lattice}, J. Algebra 80 (1983), no. 1, 159–163.

\bibitem[CS88]{CS88} Conway, J. H., Sloane, N. J. A.: {\em Low-dimensional lattices. IV. The mass formula}, Proc. Roy. Soc. London Ser. A 419 (1988), no. 1857, 259–286.

\bibitem[CS99]{CS99} Conway, J. H.,  Sloane, N. J. A.: {\em Sphere Packings, Lattices and Groups}, Springer-Verlag, 1999.

\bibitem[DOY21]{DOY21} Dinh, T.-C., Oguiso, K., Yu, X.: {\em Smooth complex projective rational surfaces with infinitely many real forms}, https://arxiv.org/pdf/2106.05687.pdf.

\bibitem[DS05]{DS05} Dinh, T.-C., Sibony, N.: {\em Une borne supérieure pour l’entropie topologique d’une application rationnelle}, Ann. of Math. (2) 161 (2005), no. 3, 1637–1644.

\bibitem[DK02]{DK02} Dolgachev, I., Keum, J.: {\em Birational automorphisms of quartic Hessian surfaces}, Trans. Amer. Math. Soc. 354 (2002), no. 8, 3031–3057.

\bibitem[ES13]{ES13}  Esnault, H., Srinivas, V. : {\em Algebraic versus topological entropy for surfaces over finite fields}, Osaka J. Math. {\bf 50} (2013) 827--846. 

\bibitem[Ge08]{Ge08} Gerstein, L. J.: {\em Basic Quadratic Forms}, American Mathematical Society, 2008.

\bibitem[HKL20]{HKL20} Hashimoto, K., Keum, J., Lee, K.: {\em K3 surfaces with Picard number 2}, Salem polynomials and Pell equation. J. Pure Appl. Algebra 224 (2020), no. 1, 432–443.

\bibitem[Hu16]{Hu16} Huybrechts, D.: {\em Lectures on K3 surfaces}, Cambridge Studies in Advanced Mathematics, 158. Cambridge University Press, Cambridge, 2016. xi+485 pp.

\bibitem[Ke00]{Ke00} Keum, J.: {\em A note on elliptic K3 surfaces}, Trans. Amer. Math. Soc. 352 (2000), no. 5, 2077–2086.

\bibitem[KK01]{KK01} Keum, J., Kondō, S.: {\em The automorphism groups of Kummer surfaces associated with the product of two elliptic curves} Trans. Amer. Math. Soc. 353 (2001), no. 4, 1469–1487.

\bibitem[Ko89]{Ko89} Kondō, S.: {\em Algebraic K3 surfaces with finite automorphism groups}, Nagoya Math. J. 116 (1989), 1–15.

\bibitem[LK13]{LK13} Lorch, D., Kirschmer, M.: {\em Single class genera of integral lattices}, LMS J. Comput. Math. 16 (2013), http://www.math.rwth-aachen.de/~Gabriele.Nebe/LATTICES/Classi/watson.

\bibitem[Mc02]{Mc02} McMullen, C. T.: {\em Dynamics on $K3$ surfaces: Salem numbers and Siegel disks}, J. Reine Angew. Math. {\bf 545}  (2002) 201--233.

\bibitem[Mc07]{Mc07} McMullen, C. T.: {\em Dynamics on blowups of the projective plane}, Publ. Math. Inst. Hautes \'Etudes Sci. {\bf 105} (2007) 49--89. 

\bibitem[Mc11a]{Mc11a} McMullen, C. T.: {\em K3 surfaces, entropy and glue}, J. Reine Angew. Math. 658, (2011), 1--25. 
%
\bibitem[Mc11b]{Mc11b} McMullen, C. T.: {\em Salem number/Coxeter group/K3 surface package}, doi:10.7910/DVN/29211
%
\bibitem[Mc16]{Mc16} McMullen, C. T.: {\em Automorphisms of projective K3 surfaces with minimum entropy}, Invent. Math. {\bf 203} (2016) 179--215.

\bibitem[Me21]{Me21} Mezzedimi, G.: {\em K3 surfaces of zero entropy admitting an elliptic fibration with only irreducible fibers}, J. Algebra 587 (2021), 344–389.

\bibitem[Mu88]{Mu88} Mukai, S.: {\em Finite groups of automorphisms of K3 surfaces and the Mathieu group}, Invent. Math. 94 (1988), no. 1, 183–221.

\bibitem[Ni79]{Ni79} Nikulin, V. V.: {\em Quotient-groups of groups of automorphisms of hyperbolic forms of subgroups generated by 2-reflections},
Dokl. Akad. Nauk SSSR 248 (1979), no. 6, 1307–1309.

\bibitem[Ni80]{Ni80} Nikulin,V.V.: {\em Integral symmetric bilinear forms and some of their applications}, Math.USSRIzv. {\bf 14} (1980) 103--167.

\bibitem[Ni81]{Ni81} Nikulin, V. V.: {\em Quotient-groups of groups of automorphisms of hyperbolic forms by subgroups generated by 2-reflections. Algebro-geometric Applications}, Current Problems in Mathematics, vol. 18, Akad. Nauk SSSR, Vsesoyuz. Inst. Nauchn. i Tekhn. Informatsii, Moscow, 1981. 3–114. English transl. in: J. Soviet Math. 22 (1983), no. 4, 1401–1475.

\bibitem[Ni85]{Ni85} Nikulin, V. V.: {\em Surfaces of type K3 with a finite automorphism group and a Picard
group of rank three}, Proc. Steklov Institute of Math. Issue, 3 (1985), 131–155.

\bibitem[Ni96]{Ni96} Nikulin, V. V.: {\em Reflection groups in Lobachevsky spaces and the denominator identity for Lorentzian Kac–Moody algebras}, Izv. Math. 60 (1996), 305–334.

\bibitem[Ni99]{Ni99} Nikulin, V. V.: {\em K3 surfaces with interesting groups of automorphisms}, J. Math. Sci. 95 (1999) 2028--2048.

\bibitem[Ni14]{Ni14} Nikulin, V. V.: {\em Elliptic fibrations on K3 surfaces}, Proc. Edinb. Math. Soc. (2), 57 (2014), 253–267.

\bibitem[Ni20]{Ni20} Nikulin, V. V.: {\em Some examples of K3 surfaces with infinite automorphism group which preserves an elliptic pencil}, Math. Notes 108 (2020), no. 3-4, 542–549.

\bibitem[Og07]{Og07} Oguiso, K.: {\em Automorphisms of hyperk\"ahler manifolds in the view of topological entropy}, Algebraic geometry, Contemp. Math. {\bf 422} (2007) 173--185.
%

\bibitem[Og10a]{Og10a} Oguiso, K.: {\em The third smallest Salem number in automorphisms of K3 surfaces}, Algebraic geometry in East Asia--Seoul 2008, Adv. Stud. Pure Math., {\bf 60} (2010) 331--360.

\bibitem[Og10b]{Og10b} Oguiso, K.: {\em Salem polynomials and the bimeromorphic automorphism group of a hyper-K\"ahler manifold},  Selected papers on analysis and differential equations, Amer. Math. Soc. Transl. Ser. {\bf 230} (2010) 201--227.

\bibitem[Og14]{Og14} Oguiso, K.: {\em Some aspects of explicit birational geometry inspired by complex dynamics}, Proceedings of the International Congress of Mathematicians—Seoul 2014. Vol. II, 695–721, Kyung Moon Sa, Seoul, 2014.

\bibitem[Og15]{Og15} Oguiso, K.: {\em Free automorphisms of positive entropy on smooth Kähler surfaces}, Algebraic geometry in east Asia—Taipei 2011, 187–199,
Adv. Stud. Pure Math., 65, Math. Soc. Japan, Tokyo, 2015.

\bibitem[OY19]{OY19} Oguiso, K., Yu, X.: {\em Coble's question and complex dynamics of inertia groups on surfaces}, https://arxiv.org/pdf/1904.00175.pdf.

\bibitem[OY20]{OY20} Oguiso, K., Yu, X.: {\em Minimum positive entropy of complex Enriques surface automorphisms}, Duke Math. J. 169 (2020), no. 18, 3565–3606.

\bibitem[PS71]{PS71} Piateski-Shapiro, A., Shafarevich, I.: {\em Torelli's theorem for algebraic surfaces of type K3}, Izv. Akad. Nauk SSSR Ser. Mat. 35 (1971), 530–572.

\bibitem[Ro20]{Ro20} Roulleau, X.: {\em An atlas of K3 surfaces with finite automorphism group}, https://arxiv.org/pdf/2003.08985.pdf.

\bibitem[Sh15]{Sh15} Shimada, I.: {\em An algorithm to compute automorphism groups of K3 surfaces and an application to singular K3 surfaces}, Int. Math. Res. Not. IMRN 2015, no. 22, 11961–12014.

\bibitem[SI77]{SI77} Shioda, T., Inose, H.: {\em On singular K3 surfaces}, Complex analysis and algebraic geometry, pp. 119–136. Iwanami Shoten, Tokyo, 1977.

\bibitem[Th]{Th} The PARI Group: {\em PARI/GP version 2.7.5}, Bordeaux (2015) http://pari.math.u-bordeaux.fr/
%
\bibitem[The]{The} The Sage Developers: {\em SageMath, the Sage Mathematics Software System (Version 7.2)}, 2016. https://www.sagemath.org.

\bibitem[Vi75]{Vi75} Vinberg, E. B.: {\em Some Arithmetical Discrete Groups in Lobačevskiĭ spaces}, In Discrete Subgroups of Lie Groups and Applications to Moduli (Internat. Colloq., Bombay, 1973), 323–48. Bombay: Oxford University Press, 1975.

\bibitem[Vi83]{Vi83} Vinberg, E. B.: {\em The two most algebraic K3 surfaces},
Math. Ann. 265 (1983), no. 1, 1–21.

\bibitem[Vi07]{Vi07} Vinberg, E. B.: {\em Classification of 2-reflective hyperbolic lattices of rank 4}, Trans. Moscow Math. Soc. 2007, 39–66.

\bibitem[Vo07]{Vo07} Voight, J.: {\em Quadratic forms that represent almost the same primes}, Math. Comp.  76 (2007), 1589–1617.

\bibitem[Wa63]{Wa63} Watson, G. L.: {\em The class-number of a positive quadratic form}, Proc. London Math. Soc. (3) 13 (1963), 549–576.

\bibitem[Wo]{Wo} Wolfram Research, Inc.: {\em Mathematica (Version 10.0)}, Champaign, IL (2014).

\bibitem[Yu18]{Yu18} Yu, X.: {\em Elliptic fibrations on K3 surfaces and Salem numbers of maximal degree}, J. Math. Soc. Japan, {\bf 70} (2018) 1151-1163.
 
 
\end{thebibliography}
\end{document}